\documentclass[12pt]{amsart}

\usepackage{amssymb}
\usepackage{amscd}
\usepackage{amsthm}
\usepackage{type1cm}
\usepackage{mathrsfs}
\usepackage{mathtools}
\usepackage[a4paper]{geometry}
\usepackage[driverfallback=dvipdfm]{hyperref}
\usepackage{cleveref}
\usepackage{autonum}

\newgeometry{left=30mm,right=30mm,bottom=32mm,top=32mm}

\newtheorem{Def}{Definition}
\newtheorem*{Def*}{Definition}
\newtheorem{thm}{Theorem}
\newtheorem*{thm*}{Theorem}
\newtheorem{thmA}{Theorem}

\newtheorem{lem}{Lemma}[section]

\newtheorem{prop}{Proposition}[section]

\theoremstyle{remark}

\theoremstyle{definition}
\newtheorem{ex}{Example}[section]

\crefname{theorem}{Theorem}{Theorems}
\crefname{lemma}{Lemma}{Lemmas}
\crefname{proposition}{Proposition}{Propositions}

\allowdisplaybreaks[4]

\newcommand{\Lip}{\operatorname{Lip}}
\newcommand{\lip}{\operatorname{lip}}
\newcommand{\ran}{\operatorname{ran}}

\begin{document}
	
	\title[]
	{On the second dual space of Banach space of vector-valued little Lipschitz functions}
	
	\author{Shinnosuke Izumi}
	
	\address{Department of Mathematics, Shinshu University, Matsumoto 390-8621, Japan}
	
	\email{izumi@math.shinshu-u.ac.jp}
	
	
\begin{abstract}
	Let \(X\) be a compact metric space 
	and \(E\) be a Banach space. 
	\(\lip (X, E)\) denotes the Banach space of all \(E\)-valued little Lipschitz functions on \(X\).  
	We show that 
	\(\lip (X, E)^{**}\) 
	is isometrically isomorphic to 
	Banach space of \(E^{**}\)-valued Lipschitz functions \(\Lip(X, E^{**})\) 
	under several conditions. 
	Moreover, 
	we describe the isometric isomorphism from \(\lip (X, E)^{**}\) to \(\Lip (X, E^{**})\).  
\end{abstract}
	\maketitle

\section{Introduction}
Let \(X\) be a compact metric space with metric \(d\), 
\(\mathbb{K}\) be a field of real number or complex number,  
\(E\) be a Banach space over \(\mathbb{K}\) with norm \(\| \cdot \|_{E}\) 
and \(0 < \alpha \le 1\) .
\(C(X, E)\) denotes the Banach space of all \(E\)-valued continuous functions on \(X\) 
with norm
\[
\| f \|_{C(X, E)} 
\coloneqq  
\sup_{x \in X} \| f(x) \|_{E}.
\]
If \(f \in C(X, E)\) satisfies the condition 
\[
\mathcal{L}_{X, E} (f) 
\coloneqq 
\sup_{\begin{subarray}{c} x, x' \in X \\ x \neq x'\end{subarray}} 
\frac{\|f(x) - f(x')\|_{E}}{d (x, x')^{\alpha}} 
< \infty, 
\]
then we say that \(f\) is \textsl{\(\alpha\)-Lipschitz}. 
In particular, 
if \(\alpha = 1\), 
then we say that \(f\) is \textsl{Lipschitz}.
\(\Lip_{\alpha} (X, E)\) denotes the set of all \(E\)-valued \(\alpha\)-Lipschitz functions on \(X\).
If \(f \in \Lip_{\alpha} (X, E)\) satisfies the condition 
\[
\lim_{d (x, x')^{\alpha} \rightarrow 0} 
\frac{\|f(x) - f(x')\|_{E}}{d (x, x')^{\alpha}} = 0, 
\]
then we say that \(f\) is \textsl{little \(\alpha\)-Lipschitz}. 
In particular, 
if \(\alpha = 1\), 
then we say that \(f\) is \textsl{little Lipschitz}. 
\(\lip_{\alpha} (X, E)\) denotes the set of all \(E\)-valued little \(\alpha\)-Lipschitz functions on \(X\).
In case that \(E = \mathbb{K}\), 
we simply write \(\Lip_{\alpha} (X) \coloneqq \Lip_{\alpha} (X, \mathbb{K})\) 
and \(\lip_{\alpha} (X) \coloneqq \lip_{\alpha} (X, \mathbb{K})\), respectively.
In case that \(\alpha = 1\), 
we simply write \(\Lip (X, E) \coloneqq \Lip_{1} (X, E)\) 
and \(\lip (X, E) \coloneqq \Lip_{1} (X, E)\), respectively.

Clearly, 
\(\Lip_{\alpha}(X, E)\) is a Banach space over \(\mathbb{K}\) with respect to the norm
\[
\|f\|_{\textrm{max}} 
\coloneqq  
\max\left\{\|f\|_{C(X, E)}, \mathcal{L}_{X, E}(f)\right\} \qquad (f \in \Lip(X, E)), 
\]
and \(\lip_{\alpha}(X, E)\) is a closed subspace of \(\Lip_{\alpha}(X, E)\). 
Moreover, 
if \(0 < \alpha < 1\), 
then \(\Lip(X, E) \subset \lip_{\alpha}(X, E)\). 

For any Banach space \(B\), 
\(B^{*}\) and \(B^{**}\) denote the dual space and second dual space of \(B\), 
respectively. 
Also, 
\(b^{*}\) and \(b^{**}\) denote the elements of \(B^{*}\) and \(B^{**}\), 
respectively. 

N. Weaver defined in \cite{We} the following property 
to get a Stone-Weierstrass type theorem. 

\begin{Def}[Weaver\cite{We}]
	Let \(X\) be a compact metric space with metric \(d\). 
	We say that \(\lip (X)\) \textsl{separates points uniformly} 
	if there exists a constant \(c > 1\) 
	such that 
	for any \(x, x' \in X\) with \(x \neq x'\), 
	some \(f \in \lip (X)\) 
	satisfies \(\mathcal{L}_{X, \mathbb{K}}(f) \le c\) 
	and \(|f (x) - f (x')| = d (x, x')\).  
\end{Def}

In \cite{BCD}, 
W. G. Bade, P. C. Curtis Jr. and H. G. Dales showed that 
\(\lip_{\alpha} (X)\) separates points uniformly for \(0 < \alpha <1\). 

L. G. Hanin \cite{Han} and N. Weaver \cite{We} showed following theorem: 

\begin{thmA}[Hanin\cite{Han} and Weaver\cite{We}]\label{ThmHW}
	Let \(X\) be a compact metric space. 
	The following are equivalent: 
	\begin{enumerate}
		\renewcommand{\labelenumi}{(\alph{enumi})}
		\item 
		\(\lip (X)\) separates points uniformly. 
		
		\item 
		There exists \(b > 1\) 
		such that 
		for any \(g \in \Lip (X)\) and any finite subset \(S \subset X\), 
		some \(f \in \lip (X)\) satisfies \(\mathcal{L}_{X, \mathbb{K}} (f) \le b \mathcal{L}_{X, \mathbb{K}}(g)\) 
		and \(f (x) = g (x)\) for \( x \in S\). 
		
		\item 
		\(\lip (X)^{**}\) is isometrically isomorphic to \(\Lip (X)\), 
		via the mapping \(\Lambda : \lip (X)^{**} \rightarrow \Lip (X)\) defined by 
		\begin{align}
		(\Lambda f^{**})(x) = f^{**} (\tau_{x}) \qquad \bigl(f^{**} \in \lip (X)^{**}, x \in X\bigr),
		\end{align}
		where 
		\(\tau_{x}\) is evaluation functional at \(x\).
	\end{enumerate}
\end{thmA}

Also, 
J. A. Johnson \cite[Theorem 4.7]{Joh}, 
and W. G. Bade, P. C. Curtis Jr. and H. G. Dales \cite[Theorem 3.5]{BCD} 
showed that 
\(\lip_{\alpha} (X)^{**}\) 
is isometrically isomorphic to \(\Lip_{\alpha} (X)\). 

For vector-valued case, 
J. A. Johnson took up in \cite{Joh} the Banach space \(\lip_{\alpha}(X, E)\) 
and proved the following theorem:  

\begin{thmA}[Johnson{\cite{Joh}}]\label{ThmJ}
	Let \(X\) be a compact metric space with metric \(d\), 
	E be a Banach space 
	and 
	\(0 < \alpha < 1\). 
	If either \(E^{*}\) or \(\lip_{\alpha} (X)^{*}\) has approximation property, 
	then \(\lip_{\alpha} (X, E)^{**}\) is isometrically isomorphic to \(\Lip_{\alpha} (X, E^{**})\). 
\end{thmA}

In this paper, 
we take up the Banach space \(\lip (X, E)\) and get the following result:  

\begin{thm}\label{thm1}
	Let E be a Banach space 
	and  
	\(X\) be a compact metric space with metric \(d\). 
	If \(\lip(X)\) separates points uniformly 
	and 
	either \(E^{*}\) or \(\lip(X)^{*}\) has approximation property, 
	then \(\lip (X, E)^{**}\) isometrically isomorphic to \(\Lip(X, E^{**})\).
\end{thm}

\section{Preliminaries}
Through this section, 
we assume that \(X\) is a compact metric space 
and \(E\) is a Banach space with norm \(\|\cdot\|_{E}\). 

\begin{prop}\label{Prop2-1}
	Let \(e^{*} \in E^{*}\) and \(f \in \lip (X , E)\).  
	A function \(e^{*} \circ f \colon X \rightarrow \mathbb{K}\) defined by 
	\begin{align}
	(e^{*} \circ f)(x) = e^{*}(f(x)) \qquad (x \in X)
	\end{align}
	enjoys \(e^{*} \circ f \in \lip (x)\) 
	and \(\|e^{*} \circ f\|_{\textrm{max}} \le \|e^{*}\|_{E^{*}} \|f\|_{\textrm{max}}\)
\end{prop}

\begin{proof}
	Straightforward.
\end{proof}
 
The next proposition implies that \(\lip (X)^{*}\) has the Radon-Nikod\'{y}m property 
(see e.g., \cite[Corollary 5.42]{Rya}). 

\begin{prop}\label{Prop2-3}
	If \(\lip(X)\) separates points uniformly, 
	then \(\lip(X)^{*}\) is separable.
\end{prop}

\begin{proof}
	It follows from compactness of \(X\) that 
	there exists a countable dense subset \(\mathcal{C}\) of \(X\). 
	By using the assumption, 
	\(\{\tau_{x}\}_{x \in \mathcal{C}}\) is countable subset of \(\lip(X)^{*}\). 
	\(\mathcal{Q}\) denotes the subset of \(\mathbb{K}\) 
	such that real and imaginary parts are rational numbers. 
	Note that \(\mathcal{Q}\) is countable dense subset of \(\mathbb{K}\). 
	Put
	\begin{align}
	\textrm{L.h.}_{\mathcal{Q}} \{\tau_{x}\}_{x \in \mathcal{C}} 
	\coloneqq 
	\left\{ 
	\sum_{k=1}^{N} \beta_{k}\tau_{x_{k}} : 
	N \in\mathbb{N}, \beta_{k} \in \mathcal{Q}, x_{k} \in \mathcal{C}(k=1,\ldots,N)
	\right\}
\end{align} 
	and
\begin{align}
	\textrm{L.h.}_{\mathbb{K}} \{\tau_{x}\}_{x \in X} 
	\coloneqq 
	\left\{ 
	\sum_{k=1}^{N} \beta_{k}\tau_{x_{k}} : 
	N \in\mathbb{N}, \beta_{k} \in \mathbb{K}, x_{k} \in X(k=1,\ldots,N)
	\right\}. 
	\end{align}
	It follows from countability of \(\{\tau_{x}\}_{x \in \mathcal{C}}\) and \(\mathcal{Q}\) that 
	\(\textrm{L.h.}_{\mathcal{Q}} \{\tau_{x}\}_{x \in \mathcal{C}}\) is countable subset of \(\lip(X)^{*}\). 
	By using the density of \(\mathcal{C}\) and \(\mathcal{Q}\), 
	we see that 
	the closure of \(\textrm{L.h.}_{\mathcal{Q}} \{\tau_{x}\}_{x \in \mathcal{C}}\) 
	with respect to the norm \(\|\cdot\|_{\lip(X)^{*}}\)
	contains \(\textrm{L.h.}_{\mathbb{K}} \{\tau_{x}\}_{x \in X}\). 
	Since \(\textrm{L.h.}_{\mathbb{K}} \{\tau_{x}\}_{x \in X}\) is dense subset of \(\lip(X)^{*}\) 
	(see e.g., \cite[Theorem 4.5]{Joh}),  
	\(\textrm{L.h.}_{\mathcal{Q}} \{\tau_{x}\}_{x \in \mathcal{C}}\) is countable dense subset of \(\lip(X)^{*}\). 
\end{proof}

The next proposition implies that norm of \(\lip (X, E)^{*}\) is the cross norm 
on algebraic tensor product \(E^{*} \otimes \lip (X)^{*}\).

\begin{prop}\label{Prop2-2}
	Let \(e^{*} \in E^{*}\) and \(h^{*} \in \lip (X)^{*}\). 
	A functional \(e^{*} \otimes h^{*}\) on \(\lip(X, E)\) defined by
	\begin{align}\label{eq:2-2-1}
		(e^{*} \otimes h^{*}) (f) = h^{*} (e^{*} \circ f) \qquad ( f \in \lip (X, E)).
	\end{align} 
	enjoys \(e^{*} \otimes h^{*} \in \lip (X, E)^{*}\) 
	and \(\|e^{*} \otimes h^{*}\|_{\lip(X, E)^{*}} = \|e^{*}\|_{E^{*}} \|h^{*}\|_{\lip (X)^{*}}\). 
\end{prop}

\begin{proof}
	Let \(f \in \lip(X, E)\) be arbitrary. 
	By \Cref{Prop2-1}, 
	we have 
	\[
	\left|(e^{*} \otimes h^{*})(f)\right| 
	\le\|h^{*}\|_{\lip (X)^{*}} \|e^{*} \circ f\|_{\textrm{max}} 
	\le\|h^{*}\|_{\lip (X)^{*}} \|e^{*}\|_{E^{*}} \|f\|_{\textrm{max}}. 
	\]
	Therefore, 
	we have \(e^{*} \otimes h^{*} \in \lip (X, E)^{*}\) 
	and \(\|e^{*} \otimes h^{*}\|_{\lip(X, E)^{*}} \le \|e^{*}\|_{E^{*}} \|h^{*}\|_{\lip (X)^{*}}\). 
	
	Let \(e \in E\) and \(h \in \lip(X)\) be arbitrary. 
	Put \(f(x) \coloneqq h(x)e \ (x \in X)\). 
	Then, 
	\(f \in \lip(X, E)\) and \(\|f\|_{\textrm{max}} = \|e\|_{E}\|h\|_{\textrm{max}}\). 
	Since the equation \((e^{*}e)h = e^{*} \circ f\) holds on \(X\), 
	we have 
	\begin{align}
		\begin{split}
			|e^{*} e||h^{*} h| 
			&= |h^{*}((e^{*} e)h)| 
			= |h^{*}(e^{*} \circ f)| \\
			&= |(e^{*} \otimes h^{*})(f)| 
			\le \|e^{*} \otimes h^{*}\|_{\lip(X, E)^{*}} \|e\|_{E}\|h\|_{\textrm{max}}. 
		\end{split}
	\end{align} 
	Therefore, 
	we have \(\|e^{*}\|_{E^{*}} \|h^{*}\|_{\lip (X)^{*}} \le \|e^{*} \otimes h^{*}\|_{\lip(X, E)^{*}}\) 
	and \(\|e^{*} \otimes h^{*}\|_{\lip(X, E)^{*}} = \|e^{*}\|_{E^{*}} \|h^{*}\|_{\lip (X)^{*}}\).
\end{proof}

Finally, 
we prove the following proposition: 

\begin{prop}\label{Prop2-4}
	If \(\lip(X)\) separates points uniformly, 
	then \(E^{*} \otimes \lip (X)^{*}\) is dense in \(\lip (X, E)^{*}\). 
\end{prop}

This proposition can be proved tracing a method in \cite{Joh}. 
If \(\lip (X)\) separates points uniformly, 
then we have 
\begin{align}
	E^{*} \otimes \{\tau_{x}\}_{x \in X} 
	\subset E^{*} \otimes \lip (X)^{*} 
	\subset \lip (X, E)^{*}. 
\end{align}
Therefore, 
to prove the \Cref{Prop2-4}, 
we show that \(E^{*} \otimes \{\tau_{x}\}_{x \in X}\) is dense in \(\lip (X, E)^{*}\) 
if \(\lip (X)\) separates the points uniformly. 

Let us introduce several notations. 
Set 
\begin{align}
	\begin{split}
		\Delta &\coloneqq \{(x, x') \in X \times X : x = x'\}, \\ 
		W &\coloneqq (X \times X) \smallsetminus \Delta, \\ 
		K &\coloneqq X \cup W, 
	\end{split}
\end{align}
so that \(K\) is locally compact space. 
\(C_{0}(K, E)\) denotes the Banach space of \(E\)-valued continuous functions 
which vanish at infinity. 
We embed \(\lip (X, E)\) in \(C_{0}(K, E)\) in the usual way; 
\begin{align}
	\begin{split}
		\tilde{f}(x) &\coloneqq f(x)\ (x \in X), \\ 
		\tilde{f}(x, x') &\coloneqq \frac{f(x) - f(x')}{d(x, x')}\ ((x, x') \in W), \\ 
		\tilde{f}(\omega) &\coloneqq 0, 
	\end{split}
\end{align} 
where \(\infty\) is the point at infinity. 
Hence \(\lip (X, E)\) can be considered as a subspace of \(C(K_{\omega}, E)\), 
where \(K_{\omega}\) is the one-point compactification of \(K\). 

Let \(Y\) be a compact Hausdorff space. 
Then \(C(Y, E)^{*}\) is isometrically isomorphic to \(M(Y, E^{*})\),  
where \(M(Y, E^{*})\) is the Banach space of \(E^{*}\)-valued countably additive regular Borel measures on \(Y\) 
with finite variation (see e.g., \cite[Theorem 1.7.1]{CM}). 

Let \(\mu \in M(X, E^{*})\). 
A functional \(f_{\mu}^{*}\) on \(\lip (X, E)\) defined by  
\begin{align}\label{func}
	f_{\mu}^{*} f \coloneqq \int_{X} f(x) \, d\mu(x) \qquad (f \in \lip(X, E))
\end{align}
enjoys \(f_{\mu}^{*} \in \lip(X, E)^{*}\). 
Set 
\begin{align}
	\mathscr{N} 
	\coloneqq 
	\left\{
	f_{\mu}^{*} : \mu \in M(X, E^{*})\ \mbox{and}\ f_{\mu}^{*}\ \mbox{is represented as \eqref{func}} 
	\right\}. 
\end{align}
Clearly, 
if \(\lip(X)\) separates points uniformly, 
\(E^{*} \otimes \{\tau_{x}\}_{x \in X}\) is contained in \(\mathscr{N}\). 

To prove \Cref{Prop2-4}, 
we first show that following proposition: 

\begin{prop}\label{Prop2-5}
	\(\mathscr{N}\) is dense in \(\lip (X, E)^{*}\). 
\end{prop}

\begin{proof}
	Let \(f^{*} \in \lip(X, E)^{*}\) be arbitrary. 
	By using the Hahn-Banach theorem, 
	there exists \(\mu \in M(K_{\omega}, E^{*})\) 
	such that 
	\begin{align}
		f^{*}f = \int_{K_{\infty}} \tilde{f} \, d\mu \qquad (f \in \lip(X, E)). 
	\end{align}
Fix a point \(x_{0} \in X\). 
For each  \(n = 1, 2, 3, \ldots\), 
set 
\begin{align}
	X_{n} \coloneqq \{x \in X : d(x, x_{0}) \le n\} 
\end{align}
and
\begin{align}
	W_{n} \coloneqq \{(x, x') \in W : x, x' \in X_{n} \ \mbox{and} \ d(x, x') \geq 1/(n+1)\}. 
\end{align}
Put \(K_{n} \coloneqq X_{n} \cup W_{n}\). 
Then, 
\(\{K_{n}\}_{n =1}^{\infty}\) is an increasing sequence of compact sets whose union is \(K\). 
For each \(n = 1, 2, 3, \ldots\), 
put 
\begin{align}
	f_{n}^{*}f 
	\coloneqq  
	\int_{K_{n}} \tilde{f}(\boldsymbol{x}) \, d\mu(\boldsymbol{x}) \qquad (f \in \lip(X, E)). 
\end{align}
Then, 
for any \(f \in \lip(X, E)\), 
we have 
\begin{align}
	\begin{split}
		|f^{*}f - f_{n}^{*}f| 
		&= \left|\int_{K_{\omega}} \tilde{f}(\boldsymbol{x}) \, d\mu(\boldsymbol{x}) 
		- \int_{K_{n}} \tilde{f} (\boldsymbol{x}) \, d\mu(\boldsymbol{x})\right| \\
		&= \left|\int_{K_{\omega} \smallsetminus K_{n}} \tilde{f}(\boldsymbol{x}) \, d\mu(\boldsymbol{x})\right| 
		\le \int_{K_{\omega} \smallsetminus K_{n}} \|\tilde{f}(\boldsymbol{x})\|_{E} \, d|\mu|(\boldsymbol{x}) \\
		&\le |\mu|(K \smallsetminus K_{n}) \left(\sup_{\boldsymbol{x} \in K_{\omega}} \|\tilde{f}\|_{E}\right) 
		\le |\mu|(K \smallsetminus K_{n}) \|f\|_{\max}, 
	\end{split}
\end{align}
where \(|\mu|\) is variation of \(\mu\). 
Hence we have 
\begin{align}
	\|f_{n}^{*} - f^{*}\|_{\lip (X, E)^{*}} \le |\mu|(K \smallsetminus K_{n}). 
\end{align}
Since \(|\mu|\) is countably additive and \(\bigcup_{n=1}^{\infty} K_{n} = K\), 
we have \(\|f_{n}^{*} - f^{*}\|_{\lip (X, E)^{*}} \rightarrow 0\). 
We next show that \(f_{n}^{*} \in \mathscr{N}\). 
For any \(f \in \lip(X, E)\), 
we have 
\begin{align}
	\begin{split}
		f_{n}^{*} f
		&= \int_{X_{n}} f (x) \, d\mu(x) + \int_{W_{n}} \tilde{f}(x, x') \, d\mu(x, x') \\
		&= \int_{X_{n}} f(x) \, d\mu(x) 
		+ \int_{W_{n}} \frac{f(x)}{d(x, x')} \, d\mu(x, x') - \int_{W_{n}} \frac{f(x')}{d(x, x')} \, d\mu(x, x'). 
	\end{split}
\end{align}
Since \(d(x, x')\) is bounded away from zero on \(W_{n}\), 
there exist \(\mu_{1}, \mu_{2} \in M(X, E^{*})\) with support contained in \(X_{n}\) 
such that 
\begin{align}
	\int_{W_{n}} \frac{f(x)}{d(x, x')} \, d\mu(x, x') 
	= \int_{X_{n}} f (x) \, d\mu_{1}(x) \\
	\intertext{and}
	\int_{W_{n}} \frac{f(x')}{d(x, x')} \, d\mu(x, x') 
	= \int_{X_{n}} f (x) \, d\mu_{2}(x). 
\end{align} 
Hence we have 
\begin{align}
	f_{n}^{*} f = \int_{X_{n}} f(x) \, d(\mu + \mu_{1} - \mu_{2})(x).
\end{align} 
Therefore, 
we have \(f_{n}^{*} \in \mathscr{N}\). 
\end{proof}

Next, we prove the following proposition: 
\begin{prop}\label{Prop2-6}
	If \(\lip (X)\) separates points uniformly, 
	then \(E^{*} \otimes \{\tau_{x}\}_{x \in X}\) is dense in \(\mathscr{N}\). 
\end{prop}

\begin{proof}
	Let \(f_{\mu}^{*} \in \mathscr{N}\) and \(\gamma > 0\) be arbitrary. 
	\(\operatorname{supp} \mu\) denotes the support of \(\mu\). 
	By noting that \(X\) is compact 
	and that \(\operatorname{supp} \mu\) is closed in \(X\), 
	there exist \(x'_{1}, \ldots, x'_{N} \in \operatorname{supp} \mu\) 
	such that \(\operatorname{supp} \mu \subset \bigcup_{j=1}^{N} B_{j}\), 
	where \(B_{j} \coloneqq \{x \in X : d(x, x'_{j}) < \gamma/(4|\mu|(\operatorname{supp} \mu))\}\). 
	Set 
	\begin{align}
		\begin{split}
			A_{1} &\coloneqq B_{1}, \\ 
			A_{j} &\coloneqq B_{j} \smallsetminus \bigcup_{i=1}^{j-1}B_{i} \ (j=2, \ldots, N).
		\end{split}
	\end{align}
Then \(\{A_{j}\}_{j=1}^{N}\) is a disjoint collection of Borel sets such that 
\(\operatorname{supp} \mu \subset \bigcup_{j=1}^{N} A_{j}\) 
and \(A_{j} \subset B_{j} \ (j=1, \ldots, N)\). 
We may assume that \(A_{j} \neq \emptyset \ (j=1, \ldots, N)\) without loss generality. 
Then, 
for each \(j = 1, \dots, N\), 
we can choose \(x_{j} \in A_{j}\). 
Set 
\begin{align}
	f^{*} \coloneqq \sum_{j=1}^{N} \mu(A_{j}) \otimes \tau_{x_{j}}. 
\end{align}
Then \(f^{*} \in E \otimes \{\tau_{x}\}_{x \in X}\). 
Let \(f \in \lip(X, E)\) with \(\|f\|_{\max} \le 1\) be arbitrary.  
Set 
\begin{align}
	g \coloneqq \sum_{j=1}^{N} f(x_{j})\chi_{A_{j}}, 
\end{align}
where \(\chi_{A_{j}}\) is the characteristic function on \(A_{j}\). 
For any \(x \in \operatorname{supp} \mu\), 
there exists a unique \(i \in \{1, \ldots, N\}\) such that \(x \in A_{i}\). 
By noting that \(A_{i} \subset B_{i}\), 
we have 
\begin{align}
	\begin{split}
		\|g(x) - f(x)\|_{E} 
		&= \|f(x_{i}) - f(x)\|_{E} 
		\le d(x_{i}, x) \\
		&\le d(x_{i}, x'_{i}) + d(x'_{i}, x) 
		<\frac{\gamma}{2 |\mu|(\operatorname{supp} \mu)}. 
	\end{split}
\end{align}
Therefore we have 
\begin{align}
	\sup_{x \in \operatorname{supp} \mu}\|g(x) - f(x)\|_{E} \le \frac{\gamma}{2 |\mu|(\operatorname{supp} \mu)}. 
\end{align}
Since 
\begin{align}
	\int_{X} g(x) \, d\mu(x) 
	= \sum_{j=1}^{N} (\mu(A_{j})) (f(x_{j})) 
	= \sum_{j=1}^{N} (\mu(A_{j}) \otimes \tau_{x_{j}}) (f) = f^{*} f, 
\end{align}
we have 
\begin{align}
	\begin{split}
		|f^{*}f - f_{\mu}^{*}f|
		&= \left|\int_{X} g(x) \, d\mu(x) - \int_{X} f \, d\mu(x)\right| 
		= \left|\int_{X} (g - f)(x) \, d\mu(x)\right| \\
		&\le \int_{X} \|g(x) - f(x)\|_{E} \, d|\mu|(x)
		= \int_{\operatorname{supp} \mu} \|g(x) - f(x)\|_{E} \ d|\mu|(x) \\
		&\le |\mu|(\operatorname{supp} \mu) \left(\sup_{x \in \operatorname{supp} \mu}\|g(x) - f(x)\|_{E}\right) 
		\le \frac{\gamma}{2}. 
	\end{split}
\end{align}
Hence we have 
\begin{align}
	\|f^{*} - f_{\mu}^{*}\|_{\lip (X, E)^{*}} < \gamma. 
\end{align}
Therefore, 
\(E \otimes \{\tau_{x}\}_{x \in X}\) is dense in \(\mathscr{N}\). 
\end{proof}

\Cref{Prop2-5} and \Cref{Prop2-6} imply that 
\(E^{*} \otimes \{\tau_{x}\}_{x \in X}\) is dense in \(\lip (X, E)^{*}\) 
if \(\lip (X)\) separates points uniformly. 
Therefore, \(E^{*} \otimes \lip (X)^{*}\) is dense in \(\lip (X, E)^{*}\) 
if \(\lip (X)\) separates points uniformly.

\section{Proof of Main Theorem}
In this section, we prove \Cref{thm1}. 
To prove it, let us introduce several notations.    
For any Banach spaces \(B_{1}\) and \(B_{2}\), 
\(B_{1} \otimes_{\varepsilon} B_{2}\) and \(B_{1} \otimes_{\pi} B_{2}\) 
denote the algebraic tensor product of \(B_{1}\) and \(B_{2}\) 
with injective cross norm \(\varepsilon\) and projective cross norm \(\pi\), respectively. 
\(B_{1} \widehat{\otimes}_{\varepsilon} B_{2}\) 
and \(B_{1} \widehat{\otimes}_{\pi} B_{2}\) 
denote the 
completion of \(B_{1} \otimes_{\varepsilon} B_{2}\) 
and \(B_{1} \otimes_{\pi} B_{2}\) 
with respect to the cross norm \(\varepsilon\) and \(\pi\), respectively.

Also, 
let us introduce several conditions as follows: 
\begin{enumerate}
	\renewcommand{\labelenumi}{(A\arabic{enumi})}
	\item 
	\(\lip (X)\) separates the points uniformly. 
	\label{con1}
	\item
	Either \(\lip(X)^{*}\) or \(E^{*}\) has the approximation property. 
	\label{con2}
\end{enumerate}
By \Cref{Prop2-3}, 
the condition (A\ref{con1}) implies that 
\(\lip(X)^{*}\) has Radon-Nikod\'{y}m property. 

By using the \cite[Theorem 5.33]{Rya}, 
we get following lemma: 

\begin{lem}\label{Lem3-1}
	Suppose (A\ref{con1}) and (A\ref{con2}). 
	Then  
	\(E^{*} \widehat{\otimes}_{\pi} \lip(X)^{*}\) is isometrically isomorphic 
	to \(\left( E \widehat{\otimes}_{\varepsilon} \lip(X) \right)^{*}\), 
	via the map 
	\(
	T \colon E^{*} \widehat{\otimes}_{\pi} \lip(X)^{*} \rightarrow \left(E \widehat{\otimes}_{\varepsilon} \lip(X) \right)^{*}
	\) 
	defined by  
	\begin{align}\label{eq:3-1-1}
	\left(T\left(\sum_{j=1}^{N} e_j^{*} \otimes h_j^{*}\right)\right)\left(\sum_{i=1}^{M} e_{i} \otimes h_{i}\right) 
	= \sum_{j=1}^{N} \sum_{i=1}^{M} (e_{j}^{*} e_{i}) (h_{j}^{*} h_{i}) 
	\end{align}
	for any \(\sum_{j=1}^{n} e_j^{*} \otimes h_j^{*} \in E^{*} \otimes_{\pi} \lip (X)^{*}\) 
	and \(\sum_{i=1}^{m} e_{i} \otimes h_{i} \in E \otimes_{\varepsilon} \lip (X)\).
\end{lem}

By \Cref{Lem3-1}, 
\(\left( E \widehat{\otimes}_{\varepsilon} \lip(X) \right)^{**}\) is isometrically isomorphic to 
\(\left(E^{*} \widehat{\otimes}_{\pi} \lip(X)^{*}\right)^{*}\). 

\begin{lem}\label{Lem3-2}
	The mapping 
	\(U_{0} \colon E \otimes_{\varepsilon} \lip(X) \rightarrow \lip (X, E)\) 
	defined by 
	\begin{align}\label{eq:3-2-1}
	(U_{0} z) (x) \coloneqq \sum_{i=1}^{M} h_{i} (x) e_{i} 
	\qquad 
	\left(
	z = \sum_{i=1}^{M} e_{i} \otimes h_{i}
	\in E \otimes_{\varepsilon} \lip(X), 
	x \in X
	\right)
	\end{align}
	is a linear isometry. 
\end{lem}

\begin{proof}
	Let 
	\(
	z = \sum_{i=1}^{M} e_{i} \otimes h_{i}
	\in E \otimes_{\varepsilon} \lip(X)
	\). 
	Then we have 
	\begin{align}\label{eq:3-2-2}
	\begin{split}
	\|U_{0} z\|_{C(X, E)} 
	&=\sup_{x \in X} \left\| \sum_{i=1}^{M} h_{i}(x) e_{i} \right\|_{E}
	=\sup_{x \in X} 
	\sup_{\begin{subarray}{c}e^{*} \in E^{*} \\ \| e^{*} \|_{E^{*}} \le 1\end{subarray}} 
	\left| e^{*} \left( \sum_{i=1}^{M} h_{i}(x) e_{i} \right) \right| \\
	&=\sup_{\begin{subarray}{c}e^{*} \in E^{*} \\ \| e^{*} \|_{E^{*}} \le 1\end{subarray}} 
	\sup_{x \in X}
	\left|\sum_{i=1}^{M} (e^{*} e_{i}) h_{i}(x)\right|
	=\sup_{\begin{subarray}{c}e^{*} \in E^{*} \\ \| e^{*} \|_{E^{*}} \le 1\end{subarray}} 
	\left\|\sum_{i=1}^{M} (e^{*} e_{i}) h_{i}\right\|_{C(X)}
	\end{split}	
	\end{align}
	and 
	\begin{align}\label{eq:3-2-3}
	\begin{split}
	\mathcal{L}_{X, E}(U_{0} z) 
	&=\sup_{\begin{subarray}{c} x, x' \in X \\ x \neq x'\end{subarray}} 
	\frac{\left\| \sum_{i=1}^{M} h_{i}(x) e_{i} - \sum_{i=1}^{M} h_{i}(x') e_{i} \right\|_{E}}{d (x,x')} \\
	&=\sup_{\begin{subarray}{c} x, x' \in X \\ x \neq x'\end{subarray}} 
	\sup_{\begin{subarray}{c}e^{*} \in E^{*} \\ \| e^{*} \|_{E^{*}} \le 1\end{subarray}} 
	\frac{\left| e^{*} \left(\sum_{i=1}^{M} h_{i}(x) e_{i} - \sum_{i=1}^{M} h_{i}(x') e_{i} \right) \right|}
	{d (x,x')}\\
	&=\sup_{\begin{subarray}{c}e^{*} \in E^{*} \\ \| e^{*} \|_{E^{*}} \le 1\end{subarray}}
	\sup_{\begin{subarray}{c} x, x' \in X \\ x \neq x'\end{subarray}}
	\frac{\left|\sum_{i=1}^{M} (e^{*} e_{i}) h_{i}(x)-\sum_{i=1}^{M} (e^{*} e_{i}) h_{i}(x')\right|}
	{d (x,x')}\\
	&=\sup_{\begin{subarray}{c}e^{*} \in E^{*} \\ \| e^{*} \|_{E^{*}} \le 1\end{subarray}} 
	\mathcal{L}_{X, \mathbb{K}}\left(\sum_{i=1}^{M} (e^{*} e_{i}) h_{i}\right). 
	\end{split}
	\end{align}
	By using the equations \eqref{eq:3-2-2} and \eqref{eq:3-2-3}, 
	we have 
	\begin{align}
	\begin{split}
	\left\| U_{0} z \right\|_{\textrm{max}}
	&=\max\left\{
	\sup_{\begin{subarray}{c}e^{*} \in E^{*} \\ \| e^{*} \|_{E^{*}} \le 1\end{subarray}} 
	\left\| \sum_{i=1}^{M} (e^{*} e_{i}) h_{i} \right\|_{C(X)}, 
	\sup_{\begin{subarray}{c}e^{*} \in E^{*} \\ \| e^{*} \|_{E^{*}} \le 1\end{subarray}} 
	\mathcal{L}_{X, \mathbb{K}} \left( \sum_{i=1}^{M} (e^{*} e_{i}) h_{i} \right)\right\} \\
	&=\sup_{\begin{subarray}{c}e^{*} \in E^{*} \\ \| e^{*} \|_{E^{*}} \le 1\end{subarray}} 
	\max \left\{\left\| \sum_{i=1}^{M} (e^{*} e_{i}) h_{i} \right\|_{C(X)}, 
	\mathcal{L}_{X, \mathbb{K}} \left( \sum_{i=1}^{M} (e^{*} e_{i}) h_{i} \right) \right\} \\
	&= \sup_{\begin{subarray}{c}e^{*} \in E^{*} \\ \| e^{*} \|_{E^{*}} \le 1\end{subarray}} 
	\left\| \sum_{i=1}^{M} (e^{*} e_{i}) h_{i} \right\|_{\textrm{max}} 
	=\varepsilon(z).
	\end{split}
	\end{align}
\end{proof}

By noting that \(E \widehat{\otimes}_{\varepsilon} \lip(X)\) is completion of \(E \otimes_{\varepsilon} \lip(X)\), 
there exists a unique linear isometry 
\(
U \colon E \widehat{\otimes}_{\varepsilon} \lip(X) \rightarrow \lip(X, E)
\)
such that 
\begin{align}\label{eq:3-2-4}
U z = U_{0} z 
\qquad 
\left(z \in E \otimes_{\varepsilon} \lip(X)\right). 
\end{align}

\begin{lem}\label{Lem3-3}
	Suppose (A\ref{con1}) and (A\ref{con2}). 
	Then \(U\) is an isometric isomorphism 
	from \(E \widehat{\otimes}_{\varepsilon} \lip(X)\) to \(\lip(X, E)\).
\end{lem}

\begin{proof}
	It is sufficient to show that \(U\) is surjective. 
	\(\ran(U)\) denotes the range of \(U\). 
	Clearly, 
	\(\ran(U)\) is closed subspace of \(\lip(X, E)\). 
	Here, 
	we show that 
	\(\ran(U)\) is dense in \(\lip(X, E)\). 
	Note that 
	\(E^{*} \otimes \lip (X)^{*}\) is subspace of \(\left(E \widehat{\otimes}_{\varepsilon} \lip(X)\right)^{*}\) 
	(see e.g., \cite[Proposition 3.1(c)]{Rya}), 
	and that \(\|\cdot\|_{\lip(X, E)^{*}}\) is cross norm on \(E^{*} \otimes \lip (X)^{*}\) 
	by \Cref{Prop2-2}. 
	It is easy to see that  
	\begin{align}\label{ineq:Lem3-3}
		\varepsilon^{*}(z^{*}) \le \|z^{*}\|_{\lip (X, E)^{*}} \le \pi(z^{*}) 
		\qquad 
		\left(z^{*} \in E^{*} \otimes \lip (X)^{*}\right), 
	\end{align}
	where \(\varepsilon^{*}\) is dual of the injective cross norm \(\varepsilon\). 
	By \Cref{Lem3-1}, 
	we have 
	\begin{align}\label{eq:Lem3-3}
	\varepsilon^{*}(z^{*}) = \pi(z^{*}) 
	\qquad 
	\left(z^{*} \in E^{*} \otimes \lip(X)^{*}\right). 
	\end{align}
	By combining \eqref{ineq:Lem3-3} and \eqref{eq:Lem3-3}, we have 
	\begin{align}
	\|z^{*}\|_{\lip(X, E)^{*}} = \pi(z^{*})
	\qquad 
	\left(z^{*} \in E^{*} \otimes \lip(X)^{*}\right).
	\end{align}
	Therefore, the mapping  
	\begin{align}
		J_{0} : (E^{*} \otimes \lip(X)^{*}, \|\cdot\|_{\lip(X, E)^{*}}) \ni z^{*} 
		\mapsto 
		z^{*} \in E^{*} \widehat{\otimes}_{\pi} \lip(X)^{*}
	\end{align}
	is an linear isometry. 
	By using the \Cref{Prop2-4}, 
	there exists a unique linear isometry 
	\(J \colon \lip(X, E)^{*} \rightarrow E^{*} \widehat{\otimes}_{\pi} \lip(X)^{*}\)
	such that 
	\begin{align}
	J z^{*} = J_{0} z^{*} \qquad (z^{*} \in E^{*} \otimes \lip (X)^{*}).
	\end{align} 
	Let 
	\(z^{*} = \sum_{j=1}^{n} e_j^{*} \otimes h_j^{*} \in E^{*} \otimes_{\pi} \lip (X)^{*}\) 
	and \(z = \sum_{i=1}^{m} e_{i} \otimes h_{i} \in E \otimes_{\varepsilon} \lip (X)\) 
	be arbitrary.
	Then we have 
	\begin{align}
	\begin{split}
	(T J z^{*})(z) 
	&= (T z^{*})(z) 
	=\sum_{j=1}^{N} \sum_{i=1}^{M} (e_{j}^{*} e_{i}) (h_{j}^{*} h_{i}) \\
	&= \sum_{j=1}^{N} h_{j}^{*} \left(\sum_{i=1}^{M} (e_{j}^{*} e_{i}) h_{i}\right) 
	= \sum_{j=1}^{N} h_{j}^{*} (e_{j}^{*} \circ (U z)) \\
	&= \sum_{j=1}^{N} (e_{j}^{*} \otimes h_{j}^{*})(U z) 
	= z^{*} (U z)
	= (U^{*} z^{*})(z). 
	\end{split}
	\end{align}
	By noting that 
	\(E \widehat{\otimes}_{\varepsilon} \lip(X)\) is completion of \(E \otimes_{\varepsilon} \lip (X)\), 
	and that \(T J z^{*}\) and \(U^{*} z^{*}\) are bounded 
	on \(E \widehat{\otimes}_{\varepsilon} \lip (X)\), 
	the equation 
	\(
	T \widetilde{J} z^{*} = U^{*} z^{*}
	\) 
	holds on \(E \widehat{\otimes}_{\varepsilon} \lip (X)\). 
	By noting that \(E^{*} \otimes \lip (X)^{*}\) is dense in \(\lip (X, E)^{*}\) by \Cref{Prop2-4}, 
	and that \(T J\) and \(U^{*}\) are bounded on \(\lip (X, E)^{*}\), 
	the equation \(T J = U^{*}\) holds on \(\lip (X, E)^{*}\). 
	Since 
	\(
	T \colon E^{*} \widehat{\otimes}_{\pi} \lip(X)^{*} \rightarrow \left(E \widehat{\otimes}_{\varepsilon} \lip(X) \right)^{*}
	\) 
	and 
	\(
	J \colon \lip(X, E)^{*} \rightarrow E^{*} \widehat{\otimes}_{\pi} \lip(X)^{*}
	\) are isometries, 
	respectively, 
	\(
	U^{*} \colon  \lip (X, E)^{*} \rightarrow \left(E \widehat{\otimes}_{\varepsilon} \lip(X) \right)^{*} 
	\) 
	is isometry. 
	Therefore, 
	\(\ran(U)\) is dense in \(\lip(X, E)\).
\end{proof}

By \Cref{Lem3-3}, 
\(\lip(X, E)^{**}\) is isometrically isomorphic to \(\left(E \widehat{\otimes}_{\varepsilon} \lip(X)\right)^{**}\). 

For any Banach spaces \(B_{1}\) and \(B_{2}\), 
\(\mathscr{B} \left( B_{1}, B_{2} \right)\) denotes the Banach space of all bounded operators 
from \(B_{1}\) into \(B_{2}\) 
with operator norm \(\|\cdot\|_{\mathscr{B} \left( B_{1}, B_{2} \right)}\). 

\begin{lem}\label{Lem3-4}
	Suppose (A\ref{con1}). 
	Then \(\left( E^{*} \widehat{\otimes}_{\pi} \lip(X)^{*} \right)^{*}\) 
	is isometrically isomorphic to \(\mathscr{B} \left( E^{*}, \Lip(X) \right)\).
\end{lem}

\begin{proof}
	By \cite[p.24]{Rya}, 
	\(\left( E^{*} \widehat{\otimes}_{\pi} \lip(X)^{*} \right)^{*}\) is isometrically isomorphic to \(\mathscr{B} \left( E^{*}, \lip(X)^{**} \right)\), 
	via the map 
	\(\Psi \colon \left(E^{*} \widehat{\otimes}_{\pi} \lip(X)^{*}\right)^{*} \rightarrow \mathscr{B}\left( E^{*}, \lip(X)^{**}\right)\) 
	defined by
	\begin{align}
		\left((\Psi f^{**})(e^{*})\right)(h^{*}) \coloneqq f^{**}(e^{*} \otimes h^{*})
	\end{align}
	for any \(f^{**} \in \left( E^{*} \widehat{\otimes}_{\pi} \lip(X)^{*} \right)^{*}\), \(e^{*} \in E^{*}\) and \(h^{*} \in \lip(X)^{*}\).
	
	By \Cref{ThmHW}, 
	\(\mathscr{B} \left(E^{*}, \lip(X)^{**}\right)\) is 
	isometrically isomorphic to \(\mathscr{B} \left(E^{*}, \Lip(X) \right)\), 
	via the map 
	\(\Phi \colon \mathscr{B}\left(E^{*}, \lip(X)^{**}\right) \rightarrow \mathscr{B} \left(E^{*}, \Lip(X) \right)\) defined by 
	\begin{align}
		\left((\Phi A) (e^{*}) \right)(x) \coloneqq (A e^{*}) (\tau_{x})
	\end{align}
	for any \(A \in \mathscr{B}\left(E^{*}, \lip(X)^{**}\right)\), \(e^{*} \in E^{*}\) and \(x \in X\).
	
	Set \(S \coloneqq \Phi \Psi\). 
	Since \(\Phi\) and \(\Psi\) are isometric isomorphisms, 
	\(S\) is an isometric isomorphism from \(\left( E^{*} \widehat{\otimes}_{\pi} \lip(X)^{*} \right)^{*}\) 
	to \(\mathscr{B} \left( E^{*}, \Lip(X) \right)\). 
	Moreover, 
	\begin{align}\label{eq:3-4-1}
	\left((S f^{**})(e^{*})\right)(x) 
	= \left((\Psi f^{**})(e^{*})\right)(\tau_{x}) 
	= f^{**} (e^{*} \otimes \tau_{x})
	\end{align}
	for any \(f^{**} \in \left( E^{*} \widehat{\otimes}_{\pi} \lip(X)^{*} \right)^{*}\), \(e^{*} \in E^{*}\) 
	and \(x \in X\).
\end{proof}

\begin{lem}\label{Lem3-5}
	\(\mathscr{B}\left(E^{*}, \Lip(X)\right)\) is isometrically isomorphic to \(\Lip(X, E^{**})\).
\end{lem}

\begin{proof}
	The mapping \(R \colon \mathscr{B}\left(E^{*}, \Lip(X)\right) \rightarrow \Lip (X , E^{**})\) is defined by 
	\begin{align}\label{eq:3-5-1}
	\left((R A) (x) \right)(e^{*}) \coloneqq \left( A e^{*} \right)(x) 
	\qquad 
	\left(A \in \mathscr{B}\left(E^{*}, \Lip(X)\right), x \in X, e^{*} \in E^{*}\right).
	\end{align}
	Then, 
	\begin{align} 
	\begin{split}
	\|R A\|_{C(X, E^{**})} 
	=\sup_{x \in X} 
	\sup_{\begin{subarray}{c} e^{*} \in E^{*} \\ \| e^{*} \|_{E^{*}} \le 1\end{subarray}} 
	\left|(A e^{*})(x)\right|
	=\sup_{\begin{subarray}{c} e^{*} \in E^{*} \\ \| e^{*} \|_{E^{*}} \le 1\end{subarray}} 
	\| A e^{*} \|_{C(X)}
	\end{split}
	\end{align}
	and 
	\begin{align} 
	\begin{split}
	\mathcal{L}_{X, E^{**}}(R A) 
	=\sup_{\begin{subarray}{c} x, x' \in X \\ x \neq x'\end{subarray}} 
	\sup_{\begin{subarray}{c} e^{*} \in E^{*} \\ \| e^{*} \|_{E^{*}} \le 1\end{subarray}} 
	\frac{\left| (A e^{*})(x) - (A e^{*})(x') \right|}{d_X (x, x')}
	=\sup_{\begin{subarray}{c} e^{*} \in E^{*} \\ \| e^{*} \|_{E^{*}} \le 1\end{subarray}} 
	\mathcal{L}_{X, \mathbb{K}} (A e^{*}). 
	\end{split}
	\end{align}
	Therefore, we have 
	\begin{align}
	\left\|R A\right\|_{\textrm{max}} 
	&=\max\left\{ \|R A\|_{C(X, E^{**})}, \mathcal{L}_{X, E^{**}}(R A) \right\} \\
	&=\max\left\{
	\sup_{\begin{subarray}{c} e^{*} \in E^{*} \\ \| e^{*} \|_{E^{*}} \le 1\end{subarray}} \|A e^{*}\|_{C(X)}, 
	\sup_{\begin{subarray}{c} e^{*} \in E^{*} \\ \| e^{*} \|_{E^{*}} \le 1\end{subarray}} 
	\mathcal{L}_{X, \mathbb{K}} (A e^{*}) \right\} \\ 
	&=\sup_{\begin{subarray}{c} e^{*} \in E^{*} \\ \| e^{*} \|_{E^{*}} \le 1\end{subarray}}
	\max\left\{ \| A e^{*} \|_{C(X)}, \mathcal{L}_{X, \mathbb{K}} (A e^{*}) \right\} 
	=\| A \|_{\mathscr{B} (E^{*}, \Lip(X))}
	\end{align}
	Hence, \(R\) is a linear isometry. 
	Let \(g \in \Lip (X, E^{**})\). 
	The mapping \(A \colon E^{*} \rightarrow \Lip(X)\) is defined by
	\[
	\left(A e^{*}\right) (x) = \left(g(x)\right)(e^{*}) \qquad (e^{*} \in E^{*}, x \in X).
	\]
	It is clear that  
	\(A\) is a linear operator. 
	Also, 
	we have
	\begin{align}\label{eq:3-5-2}
	\begin{split}
	\sup_{\begin{subarray}{c} e^{*} \in E^{*} \\ \| e^{*} \|_{E^{*}} \le 1\end{subarray}} \| A e^{*}\|_{C(X)} 
	&= \sup_{x \in X} \sup_{\begin{subarray}{c} e^{*} \in E^{*} \\ \| e^{*} \|_{E^{*}} \le 1\end{subarray}}
	\left|(A e^{*})(x)\right| 
	= \sup_{x \in X} \sup_{\begin{subarray}{c} e^{*} \in E^{*} \\ \| e^{*} \|_{E^{*}} \le 1\end{subarray}}
	\left|\left(g(x)\right)(e^{*})\right| \\
	&= \sup_{x \in X} \| g(x) \|_{E^{**}} 
	= \|g\|_{C(X,E^{**})}
	\end{split}
	\end{align}
	and
	\begin{align}\label{eq:3-5-3}
	\begin{split}
	\sup_{\begin{subarray}{c} e^{*} \in E^{*} \\ \| e^{*} \|_{E^{*}} \le 1\end{subarray}} 
	\mathcal{L}_{X, \mathbb{K}} (A e^{*}) 
	&= \sup_{\begin{subarray}{c} x, x' \in X \\ x \neq x'\end{subarray}} 
	\sup_{\begin{subarray}{c} e^{*} \in E^{*} \\ \| e^{*} \|_{E^{*}} \le 1\end{subarray}} 
	\frac{\left| (A e^{*})(x) - (A e^{*})(x') \right|}{d_X (x, x')}  \\
	&= \sup_{\begin{subarray}{c} x, x' \in X \\ x \neq x'\end{subarray}} 
	\sup_{\begin{subarray}{c} e^{*} \in E^{*} \\ \| e^{*} \|_{E^{*}} \le 1\end{subarray}} 
	\frac{\left| \left(g(x)\right)(e^{*}) - \left(g(x')\right)(e^{*}) \right|}{d_X (x, x')} \\
	&= \sup_{\begin{subarray}{c} x, x' \in X \\ x \neq x'\end{subarray}} 
	\frac{\left\| \left(g(x)\right) - \left(g(x')\right) \right\|_{E^{**}}}{d_X (x, x')} 
	=\mathcal{L}_{X, E^{**}} (g).
	\end{split}	
	\end{align}
	The equations \eqref{eq:3-5-2} and \eqref{eq:3-5-3} show that 
	\begin{align}
	\begin{split}
	\|A\|_{\mathscr{B}\left(E^{*}, \Lip(X)\right)} 
	&= \max
	\left\{
	\sup_{\begin{subarray}{c} e^{*} \in E^{*} \\ \| e^{*} \|_{E^{*}} \le 1\end{subarray}} 
	\| A e^{*}\|_{C(X)}, 
	\sup_{\begin{subarray}{c} e^{*} \in E^{*} \\ \| e^{*} \|_{E^{*}} \le 1\end{subarray}} 
	\mathcal{L}_{X, \mathbb{K}} (A e^{*})
	\right\} \\
	&= \max\left\{\|g\|_{C(X,E^{**})}, \mathcal{L}_{X, E^{**}} (g)\right\} = \|g\|_{\textrm{max}}
	\end{split}
	\end{align}
	Therefore, 
	we have 
	\(A \in \mathscr{B} \left(E^{*}, \Lip(X)\right)\). 
	Also, we have  
	\[
	\left((R A)(x)\right)(e^{*}) = \left( A e^{*} \right)(x) = \left(g(x)\right)(e^{*}) 
	\] 
	for any \(x \in X\) and \(e^{*} \in E^{*}\). 
	Hence, \(R\) is isometric isomorphism from \(\mathscr{B}\left(E^{*}, \Lip(X)\right)\) to \(\Lip(X, E^{**})\).
\end{proof}

\begin{proof}[Proof of \Cref{thm1}] 
	Put 
	\begin{align}
	Q f^{**} \coloneqq R S T^{*} (U^{-1})^{**} f^{**} \qquad \left(f^{**} \in \lip (X, E)^{**}\right), 
	\end{align}
	where \(R\), \(S\), \(T\) and \(U\) are as in above lemmas.
	Since \(R\), \(S\), \(T^{*}\) and \((U^{-1})^{**}\) are isometric isomorphisms, 
	\(Q\) is isometric isomorphism from \(\lip (X, E)^{**}\) to \(\Lip (X, E^{**})\).
	Therefore, 
	\(\lip (X, E)^{**}\) is isometrically isomorphic to \(\Lip (X, E^{**})\).
	
	Next, 
	we compute \(Q f^{**}\) for any \(f^{**} \in \lip(X, E)^{**}\). 
	By using the equations \eqref{eq:3-4-1} and \eqref{eq:3-5-1}, 
	we have 
	\begin{align}\label{eq:1-1}
	\begin{split}
	\left( (Q f^{**})(x) \right)(e^{*}) 
	&= \left( (R S T^{*} (U^{-1})^{**} f^{**}) (x) \right)(e^{*}) 
	= \left( (S T^{*} (U^{-1})^{**} f^{**}) (e^{*}) \right)(x) \\
	&= \left(T^{*} (U^{-1})^{**} f^{**} \right)(e^{*} \otimes \tau_{x}) 
	= \left((U^{-1})^{**} f^{**} \right)\left( T (e^{*} \otimes \tau_{x}) \right) \\
	&= f^{**} \left( (U^{-1})^{*} T (e^{*} \otimes \tau_{x}) \right)
	\end{split}
	\end{align}
	for any \(f^{**} \in \lip (X, E)^{**}, x \in X\) and \(e^{*} \in E\).
	Here, 
	we compute \((U^{-1})^{*} T (e^{*} \otimes \tau_{x})\). 
	Let \(f \in \lip (X, E)\) be arbitrary. 
	Then there exists a sequence \(\{z_{k}\}_{k=1}^{\infty}\) in \(E \otimes_{\varepsilon} \lip (X)\) 
	such that 
	\begin{align}\label{eq:1-2}
	\| U z_{k} - f \|_{\textrm{max}} \rightarrow 0 \ (k \rightarrow \infty).
	\end{align}
	For each \(k = 1, 2, 3, \ldots\), 
	\(z_{k}\) is represented as follows:  
	\begin{align}
	z_{k} = \sum_{i=1}^{M_{k}} e_{i}^{(k)} \otimes h_{i}^{(k)}
	\end{align}
	with some \(e_{i}^{(k)} \in E\) and \(h_{i}^{(k)} \in \lip (X)\).
	Since \(U\) is an isometric isomorphism, 
	we have  
	\begin{align}\label{eq:1-3}
	\varepsilon(U^{-1}f -  z_{k})
	\rightarrow 0 \ (k \rightarrow \infty).
	\end{align}
	By using the equations \eqref{eq:2-2-1}, \eqref{eq:3-1-1}, \eqref{eq:3-2-1} and \eqref{eq:3-2-4}, 
	we get following estimate:  
	\begin{align}
	\begin{split}
	&\left|\left((U^{-1})^{*}T(e^{*} \otimes \tau_{x})\right)(f) - (e^{*} \otimes \tau_{x})(f)\right| \\ 
	\le &\left|\left(T (e^{*} \otimes \tau_{x} ) \right) (U^{-1} f) - \left(T(e^{*} \otimes \tau_{x}) \right)(z_{k})\right| 
	+ \left|\left(T (e^{*} \otimes \tau_{x})\right)(z_{k}) - (e^{*} \otimes \tau_{x} ) (f)\right| \\
	= &\left| \left(T(e^{*} \otimes \tau_{x}) \right)(U^{-1}f - z_{k})\right| 
	+ \left|\sum_{i=1}^{M_{k}}\left(e^{*} e_{i}^{(k)}\right)h_{i}^{(k)}(x) - e^{*}(f(x))\right| \\
	= &\left|\left( T (e^{*} \otimes \tau_{x}) \right)(U^{-1} f - z_{k})\right| 
	+ \left|e^{*}\left(\sum_{i=1}^{M_{k}} h_{i}^{(k)}(x)e_{i}^{(k)} - f(x)\right)\right| \\
	\le &\|T(e^{*} \otimes \tau_{x})\|_{\varepsilon^{*}} \varepsilon(U^{-1} f - z_{k}) 
	+ \|e^{*}\|_{E^{*}} \left\|\sum_{i=1}^{M_{k}}h_{i}^{(k)}(x)e_{i}^{(k)} - f(x) \right\|_{E} \\
	= &\pi(e^{*} \otimes \tau_{x}) \varepsilon(U^{-1} f - z_{k}) 
	+ \|e^{*}\|_{E^{*}} \left\|(U z_{k})(x) - f(x)\right\|_{E} \\ 
	\le &\|e^{*}\|_{E^{*}} \| \tau_{x} \|_{\lip(X)^{*}} \varepsilon(U^{-1} f - z_{k})
	+ \|e^{*}\|_{E^{*}} \| U z_{k}  - f \|_{\textrm{max}}. 
	\end{split}
	\end{align}
	By using \eqref{eq:1-2} and \eqref{eq:1-3}, 
	we have 
	\[
	\left( (U^{-1})^{*} T (e^{*} \otimes \tau_{x} ) \right) (f) = (e^{*} \otimes \tau_{x} ) (f).	
	\]
	Now \(f\) is arbitrary, 
	we have 
	\begin{align}\label{eq:1-4}
	(U^{-1})^{*} T (e^{*} \otimes \tau_{x} ) = e^{*} \otimes \tau_{x}.
	\end{align} 
	Therefore, 
	by combining \eqref{eq:1-1} and \eqref{eq:1-4}, 
	we have 
	\begin{align}
	\left( (Q f^{**})(x) \right)(e^{*}) = f^{**} (e^{*} \otimes \tau_{x}).
	\end{align}
\end{proof}

\section{Examples}
In this section, 
we give several examples using the \Cref{thm1}. 

\begin{ex}\label{Ex1}
	Suppose that \(0 < \alpha < 1\). 
	Then \(\lip_{\alpha} (X)\) separates points uniformly (see e.g., \cite[Lemma 3.3]{BCD}). 
	Therefore
	by \cite[Theorem 4.5]{Joh}, 
	\(\lip_{\alpha} (X)^{*}\) has the Radon-Nikod\'{y}m property. 
	It follows from \Cref{thm1} that 
	\(\lip_{\alpha}(X, E)^{**}\) is isometrically isomorphic to 
	\(\Lip_{\alpha}(X, E^{**})\) 
	if \(\lip_{\alpha}(X)^{*}\) or \(E^{*}\) has the approximation property. 
\end{ex}

\Cref{Ex1} implies that 
\Cref{thm1} is generalization of \cite[Theorem 5.13]{Joh}. 

\begin{ex}
	Suppose that \(\lip (X)\) separates points uniformly. 
	Let \(K\) be a compact Hausdorff space. 
	\(C (K)\) denotes the Banach space of \(\mathbb{K}\)-valued continuous functions on \(K\). 
	Then \(C (K)^{*}\) has the approximation property (see e.g., \cite[p.74]{Rya}). 
	It follows from \Cref{thm1} that 
	\(\lip (X, C(K))^{**}\) is isometrically isomorphic to \(\Lip (X, C(K)^{**})\).
\end{ex}

\begin{ex}
	Suppose that \(\lip (X)\) separates points uniformly. 
	Let \((\Omega, \mathfrak{M}, \mu)\) be a \(\sigma\)-finite measure space 
	and 
	let \(p \in [1, \infty]\). 
	Then \(L^{p}(\Omega)^{*}\) has the approximation property 
	(see e.g., \cite[Theorem 6.16]{Rud} and \cite[pp.73--74]{Rya}). 
	It follows from \Cref{thm1} that 
	\(\lip (X, L^{p}(\Omega))^{**}\) is isometrically isomorphic to 
	\(\Lip (X, L^{p}(\Omega)^{**})\). 
	By noting that \(L^{p}(\Omega)\) is reflexive for any \(p \in (1, \infty)\), 
	\(\lip (X, L^{p}(\Omega))^{**}\) is isometrically isomorphic to 
	\(\Lip (X, L^{p}(\Omega))\) for any \(p \in (1, \infty)\). 
\end{ex}

\begin{ex}
	Suppose that \(\lip (X)\) separates points uniformly. 
	Let \(\mathscr{H}\) be an infinite-dimensional Hilbert space. 
	A. Szankowski showed in \cite{Sza} that 
	\(\mathscr{B} (\mathscr{H})\) does not have the approximation property. 
	Therefore, 
	by \cite[Corollary 4.7]{Rya}, 
	\(\mathscr{B} (\mathscr{H})^{*}\) does not have the approximation property. 
	On the other hand, 
	if \(\Lip (X)\) has the approximation property, 
	then so does \(\lip (X)^{*}\) by \Cref{ThmHW} and \cite[Corollary 4.7]{Rya}.
	It follows from \Cref{thm1} that 
	\(\lip (X, \mathscr{B} (\mathscr{H}))^{**}\)
	is isometrically isomorphic to \(\Lip (X, \mathscr{B} (\mathscr{H})^{**})\) 
	if \(\Lip (X)\) has the approximation property. 
\end{ex}



\end{document}